\documentclass[graybox]{svmult}

\usepackage{mathptmx}       
\usepackage{helvet}         
\usepackage{courier}        
\usepackage{type1cm}        
\usepackage{makeidx}         
\usepackage{graphicx}        
\usepackage{multicol}        
\usepackage[bottom]{footmisc}

\usepackage{amsmath,amssymb}

\makeindex             

\begin{document}

\title*{Infinite Products Involving Binary Digit Sums}
\author{Samin Riasat}
\institute{Samin Riasat \at University of Waterloo, Waterloo, Ontario, Canada, \email{sriasat@uwaterloo.ca}}
%
%
\maketitle

\abstract{Let $(u_n)_{n\ge 0}$ denote the Thue-Morse sequence with values $\pm 1$. The Woods-Robbins identity below and several of its generalisations are well-known in the literature \begin{equation*}\label{WR}\prod_{n=0}^\infty\left(\frac{2n+1}{2n+2}\right)^{u_n}=\frac{1}{\sqrt 2}.\end{equation*} No other such product involving a rational function in $n$ and the sequence $u_n$ seems to be known in closed form.
  To understand these products in detail we study the function \begin{equation*}f(b,c)=\prod_{n=1}^\infty\left(\frac{n+b}{n+c}\right)^{u_n}.\end{equation*} We prove some analytical properties of $f$. We also obtain some new identities similar to the Woods-Robbins product.
}

\abstract*{Let $(u_n)_{n\ge 0}$ denote the Thue-Morse sequence with values $\pm 1$. The Woods-Robbins identity below and several of its generalisations are well-known in the literature \begin{equation*}\label{WR}\prod_{n=0}^\infty\left(\frac{2n+1}{2n+2}\right)^{u_n}=\frac{1}{\sqrt 2}.\end{equation*} No other such product involving a rational function in $n$ and the sequence $u_n$ seems to be known in closed form.
  To understand these products in detail we study the function \begin{equation*}f(b,c)=\prod_{n=1}^\infty\left(\frac{n+b}{n+c}\right)^{u_n}.\end{equation*} We prove some analytical properties of $f$. We also obtain some new identities similar to the Woods-Robbins product.
}


\section{Introduction}
\label{sec:1}



Let $s_k(n)$ denote the sum of the digits in the base-$k$ expansion of the non-negative integer $n$. Although we only consider $k=2$, our results can be easily extended to all integers $k\ge 2$. Put $u_n=(-1)^{s_2(n)}$. In other words, $u_n$ is equal to $1$ if the binary expansion of $n$ has an even number of $1$'s, and is equal to $-1$ otherwise. This is the so-called Thue-Morse sequence with values $\pm 1$. We study infinite products of the form
\begin{equation*}f(b,c):=\prod_{n = 1}^\infty\left(\frac{n+b}{n+c}\right)^{u_n}.\end{equation*}

The only known non-trivial value of $f$ (up to the relations $f(b,b)=1$ and $f(b,c)=1/f(c,b)$) seems to be 
\begin{equation*}f\left(\frac 12,1\right)=\sqrt 2,
\end{equation*}
which is the famous Woods-Robbins identity \cite{R,W}. 
Several infinite products inspired by this identity were discovered afterwards (see, e.g., \cite{AS2,ASo}), but none of them involve the sequence $u_n$. In this paper we compute another value of $f$, namely,
\begin{equation*}f\left(\frac 14,\frac 34\right)=\frac 32.\end{equation*}
In Sect.~\ref{sec:2} we look at properties of the function $f$ and introduce a related function $h$. In Sect.~\ref{sec:3} we study the analytical properties of $h$. In Sect.~\ref{sec:4} we try to find infinite products of the form $\prod R(n)^{u_n}$ admitting a closed form value, with $R$ a rational function.

This paper forms the basis for the paper \cite{ARS}. While the purpose of \cite{ARS} is to compute new products of the forms $\prod R(n)^{u_n}$ and $\prod R(n)^{t_n}$, $t_n$ being the Thue-Morse sequence with values $0,1$, we restrict ourselves in this paper to studying products of the form $\prod R(n)^{u_n}$ in greater depth.

\section{General Properties of $f$ and a New Function $h$}
\label{sec:2}

We start with the following result on convergence.

\begin{lemma}\label{num-den}
Let 
$R \in\mathbb C(x)$ be a rational function such that the values $R(n)$ are defined and non-zero for integers $n \geq 1$. 
Then, the infinite product \, $\prod_n R(n)^{u_n}$ converges if and only if the numerator and 
the denominator of $R$ have same degree and same leading coefficient.
\end{lemma}

\begin{proof} 
See \cite{ARS}, Lemma 2.1.
\qed
\end{proof}

Hence $f(b,c)$ converges for any $b,c\in\mathbb C\setminus\{-1,-2,-3,\dots\}$. Using the definition of $u_n$ we see that $f$ satisfies the following properties.

\begin{lemma}\label{f} For any $b,c,d\in\mathbb C\setminus\{-1,-2,-3,\dots\}$,
\begin{enumerate}
\item\label{p1} $f(b,b)=1$,
\item\label{p2} $f(b,c)f(c,d)=f(b,d)$,
\item\label{p3} $\displaystyle f(b,c)=\left(\frac{c+1}{b+1}\right)f\left(\frac b2,\frac c2\right)f\left(\frac{c+1}{2},\frac{b+1}{2}\right)$.
\end{enumerate}
\end{lemma}
\begin{proof} The only non-trivial claim is part \ref{p3}. To see why it is true, note that $u_{2n} = u_n$ and $u_{2n+1} = - u_n$, so that
\begin{align*} f(b,c) &= \prod_{n = 1}^\infty\left(\frac{n+b}{n+c}\right)^{u_n}\\
&=\left(\frac{1+c}{1+b}\right)\prod_{n = 1}^\infty\left(\frac{2n+b}{2n+c}\right)^{u_n}\prod_{n=1}^\infty\left(\frac{2n+1+c}{2n+1+b}\right)^{u_n}\\
&=\left(\frac{1+c}{1+b}\right)\prod_{n = 1}^\infty\left(\frac{n+\frac b2}{n+\frac c2}\right)^{u_n}\prod_{n=1}^\infty\left(\frac{n+\frac{c+1}{2}}{n+\frac{b+1}{2}}\right)^{u_n}\\
&=\left(\frac{c+1}{b+1}\right)f\left(\frac b2,\frac c2\right)f\left(\frac{c+1}{2},\frac{b+1}{2}\right)
\end{align*}
as desired.
\qed
\end{proof}
One can ask the natural question: is $f$ the unique function satisfying these properties? What if we impose some continuity/analyticity conditions?



Using the first two parts 
of Lemma~\ref{f} we get
\begin{equation*}f(b,c)f(d,e)=\frac{f(b,c)f(c,d)f(d,e)f(d,c)}{f(c,d)f(d,c)}=\frac{f(b,e)f(d,c)}{f(c,c)}=f(b,e)f(d,c).\end{equation*}
Hence the third part 
may be re-written as
\begin{equation}\label{f-h}f(b,c)=\displaystyle\frac{\displaystyle f\left(\frac b2,\frac{b+1}{2}\right)}{b+1}\bigg/ \frac{\displaystyle f\left(\frac c2,\frac{c+1}{2}\right)}{c+1}.\end{equation}
This motivates the following definition.

\begin{definition} Define the function 
\begin{equation}\label{hdef} h(x):=f\left(\frac x2,\frac{x+1}{2}\right).\end{equation}
\end{definition}
Then Eqs.~\eqref{f-h} and \eqref{hdef} give the following result.

\begin{lemma}\label{ftoh} For any $b,c\in\mathbb C\setminus\{-1,-2,-3,\dots\}$,
\begin{equation}\label{f-h-1}f(b,c)=\frac{c+1}{b+1}\cdot\frac{h(b)}{h(c)}.\end{equation}
\end{lemma}

So understanding $f$ is equivalent to understanding $h$, in the sense that each function can be completely evaluated in terms of the other. Moreover, taking $c=b+\frac 12$ in Eq.~\eqref{f-h-1} and then using Eq.~\eqref{hdef} gives the following result.

\begin{lemma}\label{fe} The function $h$ defined by Eq. \eqref{hdef} satisfies the functional equation
\begin{equation}\label{FE2} h(x)=\frac{x+1}{x+\frac 32}h\left(x+\frac 12\right)h(2x).\end{equation}
\end{lemma}
Again one may ask: is $h$ the unique solution to Eq.~\eqref{FE2}? What about monotonic/continuous/smooth solutions?

An approximate plot of $h$ is given in Fig.~\ref{hx} with the infinite product truncated at $n = 100$. 

\begin{figure}[h]
\centering
\sidecaption
\includegraphics[scale=.45]{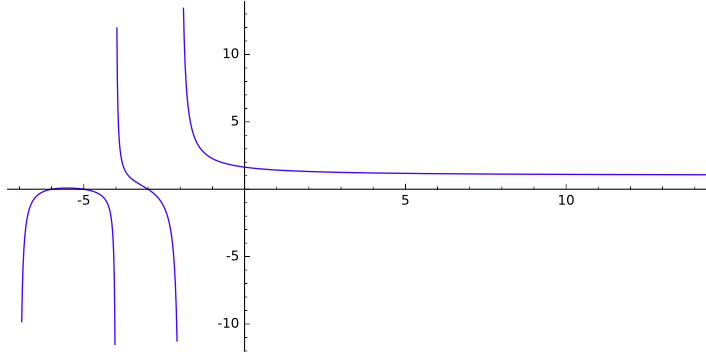}
\caption{Approximate plot of $h(x)$}
\label{hx}
\end{figure}

\section{Analytical Properties of $h$}
\label{sec:3}

The following lemma forms the basis for the results in this section.

\begin{lemma} 
\label{lem:1}
For $b,c\in(-1,\infty)$,
\begin{enumerate}
\item\label{l1} if $b=c$, then $f(b,c)=1$.
\item\label{l2} if $b>c$, then
\[\left(\frac{c+1}{b+1}\right)^2<f(b,c)<1.\]
\item\label{l3} if $b<c$, then
\[1<f(b,c)<\left(\frac{c+1}{b+1}\right)^2.\]
\end{enumerate}
\end{lemma}

\begin{proof} Using Lemma~\ref{f} it suffices to prove the second statement.

Let $b>c>-1$ 
and put 
\begin{equation}\label{notation}a_n=\log\left(\frac{n+b}{n+c}\right),\quad S_N
=\sum_{n=1}^Na_nu_n,\quad U_N=\sum_{n=1}^Nu_n.\end{equation}
Note that $a_n$ is positive and strictly decreasing to $0$. Using $s_2(2n)+s_2(2n+1)\equiv 1\pmod 2$ it follows that $U_n\in\{-2,-1,0\}$ and $U_n\equiv n\pmod 2$ for each $n$. Using summation by parts,
\begin{equation*}S_N=a_{N+1}U_N+\sum_{n=1}^NU_n(a_n-a_{n+1}).\end{equation*}
So $-2a_1<S_N<0$ for large $N$. Exponentiating and taking $N\to\infty$ gives the desired result.
\qed\end{proof}

Lemmas \ref{ftoh}-\ref{lem:1} immediately imply the following results.

\begin{theorem}\label{t} $h(x)/(x+1)$ is strictly decreasing on $(-1,\infty)$ and $h(x)(x+1)$ is strictly increasing on $(-1,\infty)$.\end{theorem}

\begin{theorem} For $b,c\in(-1,\infty)$, $f(b,c)$ is strictly decreasing in $b$ and strictly increasing in $c$.\end{theorem}

\begin{theorem}\label{hbound} For $x\in(-2,\infty)$, 
\[1<h(x)<\left(\frac{x+3}{x+2}\right)^2.\]\end{theorem}

We now give some results on differentiability.

\begin{theorem}\label{smooth} $h(x)$ is smooth on $(-2,\infty)$.\end{theorem}

\begin{proof} Recall the definition of $h$:
\begin{equation*}h(x)=\prod_{n=1}^\infty\left(\frac{2n+x}{2n+1+x}\right)^{u_n}.\]
Then taking $b=x/2$ and $c=(x+1)/2$ in Eqs.~\eqref{notation} shows that the sequence $S_n$ of smooth functions on $(-2,\infty)$ converges pointwise to $\log h$.

Differentiating with respect to $x$ gives
\[S_N'=\sum_{n=1}^N\frac{u_n}{(2n+x)(2n+1+x)}=\sum_{n=1}^Nu_n\left(\frac{1}{2n+x}-\frac{1}{2n+1+x}\right).\]
Hence
\begin{align*}\left|S_N'-S_M'\right|
&\le\sum_{n=M+1}^N\left(\frac{1}{2n+x}-\frac{1}{2n+1+x}\right)\\
&\le\sum_{n=M+1}^N\left(\frac{1}{2n-1+x}-\frac{1}{2n+1+x}\right)\\
&=\frac{1}{2M+1+x}-\frac{1}{2N+1+x}\\
&<\frac{1}{2M-1}
\to 0
\end{align*}
as $M\to\infty$, for any $x\in (-2,\infty)$ and $N>M$. Thus $S_n'$ converges uniformly on $(-2,\infty)$, which shows that $\log h$, hence $h$, is differentiable on $(-2,\infty)$.

Now suppose that derivatives of $h$ up to order $k$ exist for some $k\ge 1$. Note that
\[S_N^{(k+1)}=(-1)^kk!\sum_{n=1}^Nu_n\left(\frac{1}{(2n+x)^{k+1}}-\frac{1}{(2n+1+x)^{k+1}}\right).\]
As before,
\begin{align*}\left|S_N^{(k+1)}-S_M^{(k+1)}\right|
&\le k!\sum_{n=M+1}^N\left(\frac{1}{(2n+x)^{k+1}}-\frac{1}{(2n+1+x)^{k+1}}\right)\\
&\le k!\sum_{n=M+1}^N\left(\frac{1}{(2n-1+x)^{k+1}}-\frac{1}{(2n+1+x)^{k+1}}\right)\\
&=\frac{k!}{(2M+1+x)^{k+1}}-\frac{k!}{(2N+1+x)^{k+1}}\\
&<\frac{k!}{(2M-1)^{k+1}}
\to 0
\end{align*}
as $M\to\infty$, for any $x\in (-2,\infty)$ and $N>M$. Hence $S_n^{(k+1)}$ converges uniformly on $(-2,\infty)$, i.e., $h^{(k)}$ is differentiable on $(-2,\infty)$. 

Therefore, by induction, $h$ has derivatives of all orders on $(-2,\infty)$.
\qed\end{proof}

\begin{theorem}\label{series} Let $a\ge 0$. Then
\begin{equation*}\log h(x)=\log h(a)+\sum_{k=1}^\infty\frac{(-1)^{k-1}}{k}\left(\sum_{n=2}^\infty\frac{u_n}{(n+a)^k}\right)(x-a)^k\end{equation*}
for $x\in[a-1,a+1]$.
\end{theorem}

\begin{proof} Let $H(x)=\log h(x)$. By Theorem~\ref{smooth}, 
\[H^{(k+1)}(x)=(-1)^kk!\sum_{n=2}^\infty\frac{u_n}{(n+x)^{k+1}}.\]
Hence
\[|H^{(k+1)}(x)|\le k!\sum_{n=2}^\infty\frac{1}{|n+x|^{k+1}}\le k!\sum_{n=2}^\infty\frac{1}{(n+a-1)^{k+1}}\]
for $x\in[a-1,a+1]$. So by Taylor's inequality, the remainder for the Taylor polynomial for $H(x)$ of degree $k$ is absolutely bounded above by
\[\frac{1}{k+1}\left(\sum_{n=2}^\infty\frac{1}{(n+a-1)^{k+1}}\right)|x-a|^{k+1}\]
which tends to $0$ as $k\to\infty$, since $a\ge 0$ and $|x-a|\le 1$. Therefore $H(x)$ equals its Taylor expansion about $a$ for $x$ in the given range.
\qed\end{proof}

\section{Infinite Products}
\label{sec:4}

Recall that
\begin{equation*}f(b,c)=\prod_{n=1}^\infty\left(\frac{n+b}{n+c}\right)^{u_n}.\end{equation*}
From Lemma~\ref{f} we see that
\begin{equation}\label{3-1}\prod_{n=1}^\infty\left(\frac{(n+b)(n+\frac{b+1}{2})(n+\frac c2)}{(n+c)(n+\frac{c+1}{2})(n+\frac b2)}\right)^{u_n}=\frac{c+1}{b+1}\end{equation}
for any $b,c\neq -1,-2,-3,\dots$, and if $b,c\neq 0,-1,-2,\dots$, then
\begin{equation}\label{3-2}\prod_{n=0}^\infty\left(\frac{(n+b)(n+\frac{b+1}{2})(n+\frac c2)}{(n+c)(n+\frac{c+1}{2})(n+\frac b2)}\right)^{u_n}=1.\end{equation}
Some interesting identities can be obtained from Eqs.~\eqref{3-1}--\eqref{3-2}. For example, in Eq.~\eqref{3-1}, taking $c=(b+1)/2$ gives
\begin{equation}\label{3-3}\prod_{n=1}^\infty\left(\frac{(n+b)(n+\frac{b+1}{4})}{(n+\frac{b+3}{4})(n+\frac b2)}\right)^{u_n}=\frac{b+3}{2(b+1)}\end{equation}
while taking $b=0$ gives
\begin{equation}\label{3-4}\prod_{n=1}^\infty\left(\frac{(n+\frac 12)(n+\frac c2)}{(n+c)(n+\frac{c+1}{2})}\right)^{u_n}=c+1\end{equation}
for any $b,c\neq -1,-2,-3,\dots$. 

We now turn our attention to the functional equation Eq.~\eqref{FE2}. Recall that it reads
\begin{equation*}h(x)=\frac{x+1}{x+\frac 32}h\left(x+\frac 12\right)h(2x).\end{equation*}
Taking $x=0$ gives
\begin{equation*}h(0)=\frac 23h\left(\frac 12\right)h(0).\end{equation*}
Since $1<h(0)<9/4$ by Theorem~\ref{hbound}, cancelling $h(0)$ from both sides gives $h(1/2)=3/2$. This shows that
\begin{equation}\label{b=0}\prod_{n=0}^\infty\left(\frac{4n+3}{4n+1}\right)^{u_n}=2.\end{equation}
Next, taking $x=1/2$ in Eq.~\eqref{FE2} gives
\begin{equation*}h\left(\frac 12\right)=\frac 34h(1)^2\end{equation*}
hence $h(1)=\sqrt 2$ (since $1<h(1)<16/9$ by Theorem~\ref{hbound}) and we recover the Woods-Robbins product
\begin{equation}\prod_{n=0}^\infty\left(\frac{2n+2}{2n+1}\right)^{u_n}=\sqrt{2}.\end{equation}
Similarly, taking $x=-1/2$ in Eq.~\eqref{FE2} gives 
\begin{equation*}h\left(-\frac 12\right)=\frac 12h(0)h(-1)=\frac 12f\left(0,\frac 12\right)f\left(-\frac 12,0\right)=\frac 12f\left(-\frac 12,\frac 12\right),\end{equation*}
i.e.,
\begin{equation}\label{b=-1/2}\prod_{n=1}^\infty\left(\frac{(4n-1)(2n+1)}{(4n+1)(2n-1)}\right)^{u_n}=\frac 12.\end{equation}
Taking $x=1$ in Eq.~\eqref{FE2} gives
\begin{equation*}h(1)=\frac 45h\left(\frac 32\right)h(2)\end{equation*}
hence $h(3/2)h(2)=5\sqrt 2/4$ and this gives
\begin{equation}\label{b=1}\prod_{n=0}^\infty\left(\frac{(4n+3)(2n+2)}{(4n+5)(2n+3)}\right)^{u_n}=\frac{1}{\sqrt 2}.\end{equation}
Taking $x=3/2$ in Eq.~\eqref{FE2} and using the previous result gives 
\begin{equation*}h(2)^2h(3)=\frac{3}{\sqrt 2}\end{equation*}
which is equivalent to
\begin{equation}\label{b=3/2}\prod_{n=0}^\infty\left(\frac{(2n+2)(n+1)}{(2n+3)(n+2)}\right)^{u_n}=\frac{1}{\sqrt 2}.\end{equation}

Eqs.~\eqref{b=0}--\eqref{b=3/2} can also be combined in pairs to obtain other identities.

\section{Concluding Remarks}
\label{sec:5}

The quantity $h(0)\approx 1.62816$ appears to be of interest \cite{A,AS1}. It is not known whether its value is irrational or transcendental. We give the following explanation as to why $h(0)$ might behave specially in a sense.

Note that the only way non-trivial cancellation occurs in the functional equation Eq.~\eqref{FE2} is when $b=0$. Likewise, non-trivial cancellation occurs in Eq.~\eqref{f-h} or property \ref{p3} in Lemma~\ref{f} only for $(b,c)=(0,1/2)$ and $(1/2,0)$. That is, the victim of any such cancellation is always $h(0)$ or $h(0)^{-1}$. So one must look for other ways to understand $h(0)$. 

Using the only two known values $h(1/2)=3/2$ and $h(1)=\sqrt 2$, the following expressions for $h(0)$ can be obtained from Theorem~\ref{series}.
\begin{itemize}
\item By taking $x=0$ and $a=1$,
\[h(0)=\sqrt 2\exp\left(-\sum_{k=1}^\infty\frac 1k\sum_{n=2}^\infty\frac{u_n}{(n+1)^k}\right).\]
\item By taking $x=1$ and $a=0$,
\[h(0)=\sqrt 2\exp\left(\sum_{k=1}^\infty\frac{(-1)^k}{k}\sum_{n=2}^\infty\frac{u_n}{n^k}\right).\]
\item By taking $x=0$ and $a=1/2$,
\[h(0)=\frac 32\exp\left(\sum_{k=1}^\infty\frac 1k\sum_{n=2}^\infty\frac{u_{2n+1}}{(2n+1)^k}\right)
.\]
\item By taking $x=1/2$ and $a=0$,
\[h(0)=\frac 32\exp\left(\sum_{k=1}^\infty\frac{(-1)^k}{k}\sum_{n=2}^\infty\frac{u_{2n}}{(2n)^k}\right)
.\]
\end{itemize}
The Dirichlet series
\[\sum_{n=0}^\infty\frac{u_n}{(n+1)^k}\quad\hbox{and}\quad\sum_{n=1}^\infty\frac{u_n}{n^k}\]
appearing in the above expressions were studied by Allouche and Cohen \cite{AC}.


\begin{acknowledgement}
This work is part of a larger joint work \cite{ARS} with Professors Jean-Paul Allouche and Jeffrey Shallit. I thank the professors for helpful discussions and comments. I also thank the anonymous referees for their feedback.
\end{acknowledgement}

\end{document}